\documentclass[11pt]{amsart}
\usepackage{}
\usepackage{amssymb}
\theoremstyle{plain}
\newtheorem{Thm}{Theorem}

\newtheorem{Lem}[Thm]{Lemma}
\newtheorem{Def}[Thm]{Definition}

\begin{document}

\title[Liouville theorems and Fujita exponent]
{ Liouville theorems and Fujita exponent for nonlinear space fractional diffusions}

\author{Li Ma}
\address{Li Ma, School of Mathematics and Physics \\
  University of Science and Technology Beijing \\
  30 Xueyuan Road, Haidian District
  Beijing, 100083\\
}

\address{ Department of mathematics \\
Henan Normal university \\
Xinxiang, 453007 \\
China}

\email{lma@tsinghua.edu.cn}

\thanks{The research is partially supported by the National Natural Science
Foundation of China (No.11271111)}

\begin{abstract}
We consider non-negative solutions to the semilinear space-fractional diffusion problem $(\partial_t+(-\Delta)^{\alpha/2})u=\rho(x)u^p$ on whole space
$R^n$ with nonnegative initial data and with $(-\Delta)^{\alpha/2}$ being the $\alpha$-Laplacian operator, $\alpha\in (0,2)$. Here $p>0$ and $\rho(x)$ is a non-negative locally integrable function. For $\rho(x)=1$ we show that the fujita exponent is $p_F=1+\frac{\alpha}{n}$ and the Liouville type result for the stationary equation is true for $0<p\leq 1+\frac{\alpha}{n-\alpha}$. When $p=1/2$ and $\rho(x)$ satisfies an integrable condition, there is at least one positive solution. This existence result is proved after we establish a uniqueness result about solutions of fractional Poisson equation.


\end{abstract}

 \maketitle

\section{Introduction}
In this paper we consider blow-up property of the nonlinear fractional space-diffusion evolution equation
\begin{equation}\label{diffusion}
D_{0+}^\tau h(u)+(-\Delta)^{\alpha/2} u=f(x,u) \ \ \ in \ \ R^n\times (0,T),\ \ \ T>0
\end{equation}
and existence and non-existence results of positive solutions of some of its stationary case. Here $\tau>0$ and $0<\alpha<n$, $f$ is a nonnegative function and the problem is assumed to have reasonable non-negative initial data $u(x,0)$. The behavior of the problem (\ref{diffusion}) depends heavily on the fractional orders and the shape of the nonlinearity term $f(x,u)$. We shall consider only some special cases.
This kind of fractional-order diffusion equation models arises naturally from Anomalous diffusion processes in complex media. The time derivative term corresponds to long-time heavy tail decay and the spatial derivative for diffusion nonlocality \cite{AK} \cite{MLP} \cite{GM} \cite{MVV}. Fractional operators often cause more difficulty because of this nonlocallity.

When there is no space variables, the problem is reduced to
$$
D_{0+}^\tau h(u)=f(x,u) \ \ \ in \ \ t\in (0,T),\ \ \ T>0
$$
One may refer to \cite{BCLS} for surprising results when $\tau\in (1,2)$ and the article \cite{ABB} for a survey.

When $D_{0+}^\tau h(u)$ in (\ref{diffusion}) is replaced by the Caputo derivative $\mathbb{D}^\tau_t$, the problem is largely open and need investigation. M.Allen, L.Caffarelli, and A.Vasseur \cite{ACV} have obtained the interesting existence and regularity of a weak solution of the following fractional Porous-media flow problem
$$
\mathbb{D}^\tau_t u+(-\Delta)^{\alpha/2} u=f(x,u) \ \ \ in \ \ R^n\times (0,T),\ \ \ T>0.
$$

Recall that that for $\tau>0$, $D_{0+}^\tau h(t)$ is the Riemann-Liouville derivative of the function $h(t)$ \cite{SKM} \cite{CM} defined by
$$
D_{0+}^\tau h(t)=\frac{d^k}{dt^k}[I^{k-\tau}_{0+}h(t)]
$$
where $k=[\tau]+1$ and
$$
I^{k-\tau}_{0+}h(t)=\frac{1}{\Gamma(k-\tau)}\int_0^t\frac{h(s)}{(t-s)^{1-k+\tau}}ds
$$
is the Riemann-Liouville integral of the function $h(t)$ in $(0,T)$.
Formally, for $u\in L_{loc}(R^n)$, $(-\Delta)^{\alpha/2}u$ is the $\alpha$-Laplacian operator of the function $u:R^n\to R$ is in the distributional sense or is defined by
$$
(-\Delta)^{\alpha/2}u=C_{n,\alpha}P.V. \int_{R^n}\frac{u(x)-u(y)}{|x-y|^{n+\alpha}}dy,
$$
where $u\in C^{1,1}_{loc}(R^n)$ plus an weighted integrable condition, P.V. means the principal value of the integral under consideration.
For precise definition of latter $\alpha$-Laplacian operator of the regular functions $u:R^n\to R$, one may refer to \cite{CLL} for the precise expression for the constant $C_{n,\alpha} $.

We shall consider the important cases of (\ref{diffusion}) when $\tau=1$, $0<\alpha<2$, $h(u)=u$, $f(x,u)=\rho(x) u^p$ where $p>0$. Then the evolution equation is reduced to
\begin{equation}\label{diffusion2}
\partial_t u+(-\Delta)^{\alpha/2} u=\rho(x)u^p \ \ \ in \ \ R^n\times (0,T),\ \ \ T>0
\end{equation}
where $\rho(x)$ is a nontrivial nonnegative locally integrable function on $R^n$. It is well-known that the problem is locally well-posed in the function space $L^1\bigcap L^\infty(R^n)$ for $1<p<\infty$ (see \cite{QP} for classical case when $\alpha=2$). We shall look for the Fujita exponent in this function space. We may further reduce the problem by using the semigroup \cite{P} generated by $\alpha$-Laplacian operator. So it is quite interesting for us to consider the Fujita exponent of the nonlocal nonlinear fractional equation (\ref{diffusion2}). We show in section \ref{sect5} below that $p=p_F:==1+\frac{\alpha}{n}$ is the Fujita exponent of (\ref{diffusion2}) with $\rho(x)=1$. When $\alpha=2$, this is the classical result of Fujita \cite{F66} and Weissler \cite{W81}. When $p>p_F$ and $\alpha=2$, one may see deep results about global existence of the problem (\ref{diffusion2}) and symmetry result of positive solutions of its stationary version in the survey paper of Prof. W.M.Ni \cite{N86}. For our case when $p>p_F$, we can settle up global existence result of the flow and we shall make it appear somewhere.

When $0<\alpha<2$, we need to find physically more interesting solutions such as ground state solutions (i.e., positive solution with minimum energy) to (\ref{diffusion2}). One may consult interesting paper \cite{GLM} for Liouville properties of positive solutions to the stationary version of (\ref{diffusion2}):
\begin{equation}\label{ell}
(-\Delta)^{\alpha/2} u=\rho(x)u^p \ \ \ in \ \ R^n.
\end{equation}
When $\phi(x)=0$, Chen, D'Ambrosio, and Li \cite{CDL} find very useful Liouville theorem, and their result will be used by us in section \ref{sect2}. In \cite{JW}, the authors have developed a direct method of moving planes to nonlocal problems of variable order (which enclose the problem (\ref{ell}) as a special case) in bounded domains.
Chen, Li and Li \cite{CLL} have devised
a direct method of moving planes to study the symmetry property of positive solutions to the problem (\ref{ell}) provided the function $\rho(x)$ enjoys the symmetry and $p>1$. Chen, Li, and Li \cite{CLL2017} have further developed this kind of method applied to fully nonlinear fractional problems. Based on the work \cite{CLL}, we have developed in \cite{LM} the direct method of moving planes to positive solutions of nonlinear fractional elliptic system. Li \cite{Li} applies this method to (\ref{ell}) in unbounded parabolic domains.

When $p=0$ in (\ref{ell}), the problem is reduced to the classical Poisson equation and we shall show in section \ref{sect5} that there is a sufficient and necessary condition of solvability of Poisson equation for the problem (\ref{ell}) when $p=1/2$. We leave the general case when $p\in (0,1/2)\bigsqcup (1/2,1)$ open. We point out that when $\alpha=2$, this problem has been studied by H.Brezis and Kamin \cite{BK}.

As in the classical case when $\alpha=2$, the exponent $p_{sg}=\frac{n}{n-\alpha}$ plays a special role in understanding the Liouville property of (\ref{ell}) for $\alpha\in (0,2)$, $0<p\leq p_{sg}$ and $\rho(x)=1$. This is done in section \ref{sect4}. This kind results are even true for higher order differential operators and for elliptic systems, see \cite{LL}.

The plan od the paper is below. In section \ref{sect2}, we set up the uniqueness result for positive solutions to sublinear fractional Poisson equation in bounded domains. This uniqueness result is used to prove the existence of positive solutions to the nonlinear Poisson equation on whole space in section \ref{sect3}. It is here we show that the property the function $\rho(x)$ plays the role. In section \ref{sect4}, we prove the Liouville type theorem for the superlinear fractional Poisson equation for $0<p\leq 1+\frac{\alpha}{n-\alpha}$. In section \ref{sect5}, we make the fujita exponent from the nonlinear space-diffusion problem on whole space. We discuss possible solvable questions in section \ref{sect6}.

\section{Poisson equation of the fractional Laplacian}\label{sect2}
Let $\alpha\in (0,2)$. Define the function space
$$
\mathbf{L}_\alpha(R^n)=\{u\in L_{loc}(R^n); \int_{R^n}\frac{|u(x)|}{1+|x|^{n+\alpha}}dx<\infty\}.
$$
Clearly $L^1+L^\infty(R^n)\subset \mathbf{L}_\alpha(R^n)$.

 For $u\in \mathbf{L}_\alpha(R^n)\bigcap C^{1,1}_{loc}(R^n)$, we define
$$
Lu(x)=(-\Delta)^{\alpha/2}u=C_{n,\alpha}\lim_{\delta\to 0} \int_{R^n-B_\delta (x)}\frac{u(x)-u(y)}{|x-y|^{n+\alpha}}dy.
$$
Here $C_{n,\alpha}>0$ is the dimension constant.

Let $f$ be a locally bounded non-negative function in $R^n$, i.e, $f\in L^\infty_{loc}(R^n)$. We consider the solution to the Poisson equation
\begin{equation}\label{Pois}
Lu=f \ \ \ in \ R^n.
\end{equation}
\begin{Lem} The minimal nonnegative solution to (\ref{Pois}) is given by the formula
$$
c_{\alpha,n}\int_{R^n}\frac{f(y)}{|x-y|^{n-\alpha}}dy,
$$
where $c_{\alpha,n}>0$ is the normalized constant.
\end{Lem}

Let $R>0$ and let $u_R$ be the solution of
\begin{align*}
Lu&=f \ \ in \ B_R \\
u&=0 \ \ on \ B_R^c.
\end{align*}
Let $G_R^\alpha(x,y)$ be the Green function on $B_R$. Then we have
$$
u_R(x)=\int_{B_R} G_R^\alpha(x,y)f(y)dy.
$$
Clearly by using the maximum principle we know that $u_R(x)$ is increasing in $R$. Let
$$
u_\infty(x)=\lim_{R\to\infty}u_R(x).
$$
Recall that the limit $G_\infty^\alpha(x,y)$ of $G_R^\alpha(x,y)$ is given by
$$
G_\infty^\alpha(x,y)=\frac{c_{\alpha,n}}{|x-y|^{n-\alpha}}.
$$
Then we have
 \begin{equation}\label{exp}
u_\infty(x)=c_{\alpha,n}\int_{R^n}\frac{f(y)}{|x-y|^{n-\alpha}}dy.
 \end{equation}
It is easy to see that either
$$
u_\infty(x)<\infty, \ \ \forall \ x \in R^n
$$
or
$$
u_\infty(x)=\infty, \ \ \forall \ x \in R^n.
$$
Assume that
$U$ be a bounded solution to (\ref{Pois}). By adding a constant, we may assume that $U$ is non-negative in $R^n$.
By the maximum principle, we know that
$u_R(x)\leq U(x)$ in $B_R$. hence
$$
U_\infty(x)=\frac{c}{|x|^{n-\alpha}}* f\leq U(x) \ \ in \  R^n.
$$
Using the Liouville type theorem of Chen et \cite{CDL}, we know that if $\lim_{x\to\infty}U(x)\to 0$, then
$U(x)=u_\infty(x)$ in $R^n$. This implies that any bounded non-negative solution to (\ref{Pois}) is the minimal solution expressed by
(\ref{exp}).

We now consider the uniqueness result of a large class of nonlinear fractional Laplacian equations on bounded domain.
Let $\Omega$ be a smooth bounded domain of $R^n$. Assume that $\rho\geq c>0$ for some constant $c$.
Let $\phi$ be a non-negative bounded function in $R^n$ and let $f:R_+\to R_+$ such that there are two constants $\sigma\in(0,1)$ and $c>0$ satisfying
$$
\overline{\lim}t^{-\sigma}f(t)=c.
$$
Assume that $u>0$ in $\Omega$ satisfies that
\begin{align}\label{uniq}
Lu&=\rho f(u) \ \ in \ \Omega \\
u&=\phi \ \ on \ \Omega^c.
\end{align}

\begin{Thm}\label{uniqueness} Under the assumptions above, the problem (\ref{uniq}) has unique solution.
\end{Thm}

\begin{proof}
Suppose we have two solutions $w_1$ and $w_2$. Define
$$
\mathbb{A}=\{\lambda\in [0,1]; \lambda w_1\leq w_2 \ \ in \ \Omega\}
$$
Clearly, $0\in \mathbb{A}$ and $\mathbb{A}$ contains a neighborhood of $0$.
We claim $\lambda=1\in \mathbb{A}$. Otherwise, we have
$$
\lambda_0=\sup \mathbb{A}<1.
$$
 Let $w=w_2-\lambda w_1$. Then $w>0$ in $\Omega$.
  Then we can derive that
 \begin{equation}\label{key}
 Lw=\rho[f(w_2)-\lambda_0f(w_1)] \ \ in \ \Omega
 \end{equation}
 and
$w=(1-\lambda_0)\phi\geq 0$ on $\Omega^c$.
Note that for $\epsilon>0$ small,
$w-\epsilon w_1=(1-\lambda_0-\epsilon)\phi\geq 0$ on $\Omega^c$.
and
\begin{equation}\label{key2}
L(\epsilon w_1)=\epsilon L w_1=\rho[\epsilon f(w_1)] \ \ in \ \Omega
\end{equation}
By the equations (\ref{key}) and ({key2}) we get
\begin{equation}\label{key3}
L_K(w-\epsilon w_1)= \rho[f(w_2)-(\lambda_0+\epsilon)f(w_1)]
\end{equation}

If $w-\epsilon w_1<0$ in $\Omega$, then $w-\epsilon w_1$ attains its global minimum point $x_\epsilon\in \Omega$ and at $x_\epsilon$,
$$
L_K(w-\epsilon w_1)< 0.
$$
Sending $\epsilon $ to zero we have at $x_0=\lim_{\epsilon\to 0} x_\epsilon\in \partial\Omega$,
$w(x_0)=0$. Then $w_2(x_0)=\lambda_0 w_1(x_0)$ and so $\phi(x_0)=0$.
 At $x_\epsilon$, we have $ \lambda_0\leq \frac{w_2}{w_1}\leq \lambda_0+\epsilon$ and then
 $$
 \lim_{\epsilon\to 0} \frac{w_2}{w_1}(x_\epsilon)=\lambda_0.
 $$
By (\ref{key3}) at $x_\epsilon$, we have
 $$
 0\geq \overline{\lim}_{\epsilon\to 0} w_1(x_\epsilon)^{-\sigma}[f(w_2)-(\lambda_0+\epsilon)f(w_1)](x_\epsilon),
 $$
 which gives $0\geq c[\lambda_0^\sigma-\lambda_0]$ which is impossible for $\lambda_0\in (0,1)$.
 Then we have that $\Lambda_0+\epsilon \in \mathbb{A}$, which is impossible to the definition of $\lambda_0$. So we have $\lambda_0=1$.

This completes the proof of Theorem \ref{uniqueness}.
\end{proof}

The typical examples in application of Theorem \ref{uniqueness} are $f(u)=u^{\beta}$ and $f(u)=u^{\beta}(1-u)$. Since $u=0$ and $u=1$ are sub-solution and super solution respectively to the problem
$$
\partial_t u+(-\Delta)^{\alpha/2} u=\rho(x)u^{\beta}(1-u) \ \ \ in \ \ \Omega\times (0,T),\ \ \ T>0
$$
with initial data $u(x,0)\phi(x)\in [0,1]$, we can get a global solution to it. It is an interesting question to study the omega-limit of the flow above.

\section{Existence results for fractional Laplacian}\label{sect3}
Let $\alpha\in (0,2)$. Let $\rho(x)$ is nontrivial nonnegative function in $R^n$.

\begin{Def} We say that the function $\rho$ has the property (H) if the function
$U(x):=\int_{R^n}\frac{\rho(y)}{|x-y|^{n-\alpha}}dy$ is a bounded function in $R^n$.
\end{Def}
From the definition of $U$ we know that there is a dimension constant $c_\alpha>0$ such that
$$
(-\Delta)^{\alpha/2}U(x)=c_\alpha \rho(x) \ \ \ in \ R^n.
$$
In below, we may normalize $U$ by the constant $c_\alpha$ such that
$$
(-\Delta)^{\alpha/2}U(x)= \rho(x) \ \ \ in \ R^n.
$$

\begin{Thm}\label{existence} The problem
$$
(-\Delta)^{\alpha/2}u=\rho(x)\sqrt{u}
$$
has a bounded positive solution if and only if the function $\rho$ satisfies the property (H).

\end{Thm}

\begin{proof}
Let $\rho(x)>0$ in $R^n$ and let $u(x)>0$. Assume that
\begin{equation}\label{PM}
Lu(x)=\rho(x)\sqrt{u(x)} \ \ \ in  \ \ R^n
\end{equation}
where $L=(-\Delta)^{\alpha/2}$ is the fractional Laplacian operator on $R^n$ with $0<\alpha<2$. Let $v(x)=\sqrt{u(x)}$.
Then $u=v^2$ and
$$
\rho v=Lu=Lv^2\leq 2vLv, \ \ in \ \ R^n
$$
See \cite{M} for related derivation.
Hence by the equation (\ref{PM}) we get that
$$
\rho\leq 2Lv.
$$
That is to say, $v$ is a supersolution of the equation
\begin{equation}\label{P1}
2Lv=\rho \ \ \ in \ \ R^n.
\end{equation}
By the monotone method we then get a positive solution $\check{v}$ of (\ref{P1}) such that
$$
0<\check{v}(x)\leq \sqrt{u(x)} \ \ \ in \ R^n.
$$

\textbf{Claim}: If $U>0$ is a bounded solution of (\ref{P1}), then there is a positive solution to the equation (\ref{PM}).

For $R>0$, we can get $\phi_R$ to solve the eigenvalue problem
$$
L\phi=\lambda_{1R}\rho\phi \ \ \ in \ \ B_R
$$
and $\phi=0$ in $B_R^c$, which is the complement of the ball $B_R$. We now choose $\epsilon>0$ such that for
$\underline{u}:=\epsilon\phi$, $\lambda_{1R}\sqrt{\underline{u}}\leq 1$. Then
$$
L\underline{u}=\lambda_{1R}\rho\underline{u}\leq \rho\sqrt{\underline{u}}, \ \ \ in \ B_R.
$$

Note that
$$
L(CU)=\frac{1}{2} C\rho\geq \rho \sqrt{CU} \ \ \ in \ R^n
$$
for any fixed $C>0$ such that $\frac{1}{2}\sqrt{C}\geq \sqrt{U}$ in $R^n$. In $B_R$, we may choose $\epsilon>0$ so small that $CU\geq\underline{u}$
in $B_R$. By the monotone method we then get a positive solution $u_R$ to the R-problem
\begin{equation}\label{P2}
Lu=\rho \sqrt{u} \ \ in \ B_R
\end{equation}
with $u=0$ in $B_R^c$. This $u_R$ satisfies $\underline{u}\leq u_R(x)\leq CU(x)$ in $R^n$. By Theorem \ref{uniqueness} we know that $u_R$ is a unique positive solution to (\ref{P2}) and $u_R$ is increasing in the parameter $R$. Let $\hat{u}$ be the point-wise limit of $u_R(x)$. Then $\hat{u}>0$ is a solution to (\ref{PM}).
\end{proof}

From the proof of above result we have actually proved the following result.
\begin{Thm}\label{existence2} Assume that the function de fined $U(x):=\int_{R^n}\frac{\rho(y)}{|x-y|^{n-\alpha}}dy$ is a bounded function in $R^n$. Fix $\sigma\in (0,1)$. Then the problem
$$
(-\Delta)^{\alpha/2}u=\rho(x)u^\sigma \ \ \ in \ R^n
$$
has a bounded positive solution in $R^n$.
\end{Thm}

As an application of above ansatz, one may easily give a finite time blow up example to the fractional space time porous-media diffusion equation
\begin{equation}\label{fract}
\rho(x) D_{0+}^\beta u(x,t)=(-\Delta)^{\alpha/2}u(x,t)^2 \ \ \ in \ R^n\times [0, T).
\end{equation}
Here $\beta\in (0,1)$ and $D_{0+}^\beta$ is the Riemann-Liouville fractional derivative in $(0,\infty)$.
In fact, we set $u(x,t)=\phi(t)w(x)$, we may get for some positive constant $\lambda$ such that
\begin{equation}\label{time}
D_{0+}^\beta \phi=\lambda \phi^2 \ \ \ t>0
\end{equation}
and
$$
(-\Delta)^{\alpha/2}w(x)^2 =\lambda \rho(x)w(x) \ \ \ in \ R^n.
$$
Using the example constructed in Theorem \ref{existence} above and the fact that any positive solution to (\ref{time}) blows up in finite time, we get many finite time blow up solutions to (\ref{fract}).  The same thing can be done to
the fractional space time porous-media diffusion equation
\begin{equation}\label{fracts}
\rho(x) D_{0+}^\beta u(x,t)=(-\Delta)^{\alpha/2}u(x,t)^{1/\sigma} \ \ \ in \ R^n\times [0, T).
\end{equation}
Here $\sigma\in (0,1)$.

\section{ Liouville type result about semilinear fractional equations}\label{sect4}

In this section we consider the Liouville type result for non-negative solutions to the fractional nonlinear Poisson inequality
\begin{equation}\label{P}
(-\Delta)^{\alpha/2}u(x)\geq \rho(x)u^p \ \ in \ R^n,
\end{equation}
where $\alpha\in (0,n)$, $\rho(x)$ is a nontrivial nonnegative function with reasonable growth at $x=\infty$ and $0<p<\frac{n}{n-\alpha}$. Here, by definition, the non-negative solutions mean that they are in $L_{loc}(R^n)$ and satisfy (\ref{P}) in the distributional sense, i.e.,  for any non-negative $\phi\in C^\infty_0(R^n)$ there holds
$$
\int_{R^n} u(x)(-\Delta)^{\alpha/2}\phi(x)\geq \int_{R^n} \rho(x)u^p\phi(x).
$$

\begin{Thm}\label{liouville}
 Assume that $\rho(x)\geq 1$ in $R^n$. Let $u\geq 0$ be a distributional solution to the problem (\ref{P}). Then $u=0$.
\end{Thm}

\begin{proof}
Arguing as before, we have for almost every $x\in R^n$,
\begin{equation}\label{recursion}
u(x)\geq c \int_{R^n} \frac{\rho(y)u^p(y)}{|x-y|^{n-\alpha}}dy.
\end{equation}
We then use the well-known argument to show that $u(x)=0$ in $R^n$.

For simplicity, we first assume that $\rho(x)=1$ is the constant function. We shall use $c$ to denote various uniform constants from line to line.

Assume $u(x)$ is non-trivial in the ball $B_1$. Then we have
$$
u(x)\geq \int_{B_1}\frac{u^p}{|x-y|^{n-\alpha}}dy
$$
which implies that
$$
u(x)\geq \frac{C}{(1+|x|)^{n-\alpha}},
$$
where $C=\int_{B_1}u^p$.
We Choose $R>1$ large such that for all $|x|\geq R$,
$$
u(x)\geq \frac{C}{|x|^{n-\alpha}}.
$$
Let $D_x=\{|x-y|\leq |x|/2\}$. On $D_x$, we have $|x|/2\leq |y|\leq 3|x|/2$,
$$
u(y)\geq \frac{C}{|y|^{n-\alpha}}\geq c|x|^{\alpha-n}.
$$
Using (\ref{recursion}) we know that
$$
u(x)\geq \int_{D_x}\frac{u^p(y)}{|x-y|^{n-\alpha}}dy\geq c|x|^{\alpha-n}\int_{D_x}u^(y)dy
$$
and then for $p_1:=\alpha-n$,
$$
u(x)\geq c|x|^{pp_1+\alpha}.
$$
Let $p_2=pp_1+\alpha$. Repeat the above argument $k$-steps we have
\begin{equation}\label{key4}
u(x)\geq c|x|^{p_{k+1}},
\end{equation}
where $p_{k+1}=pp_k+\alpha$.
By induction we know that
$$
p_{k+1}=p^kp_1+\alpha(1+p+...+p^{k-1}).
$$
For $p=1$, we have
$p_{k+1}=p_1+k\alpha>0$ for some $k>1$.

Note that for $p\in (0,1)$, we have
$$
p_{k+1}=p^kp_1+\frac{\alpha(1-p^k)}{1-p}\to \frac{\alpha}{1-p}>0.
$$
Hence in finite steps we have
$p_{k}>0$.

Note that for $p\in (1, \frac{n}{n-\alpha})$, we have
$$
p_{k+1}=p^kp_1+p^k\frac{\alpha(1-m^k)}{1-m}=p^k[p_1+\frac{\alpha(1-m^k)}{1-m}]
$$
and as $k\to\infty$,
$$
p_1+\frac{\alpha(1-m^k)}{1-m}\to p_1+\frac{\alpha}{1-p}>0.
$$
Hence in finite steps we have
$p_{k}>0$.
In any case, by (\ref{key4}) and for $|x-y|\geq 2|x|$ we have
$$
|y|\geq |x-y|-|x|\geq |x-y|/2
$$
and
$$
u(x)\geq \int_{\{|x-y|\geq 2|x|\}} \frac{|y|^{pp_k}}{|x-y|^{n-\alpha}}dy=\infty,
$$
which is a contradiction to the fact $u\in L_{loc}(R^n)$.

For the case $p=\frac{n}{n-\alpha}=-\frac{n}{p_1}$ (and $pp_1=-n$), we have
$$
u(x)\geq c(R+|x|)^{p_1}\int_{B_R}u^p.
$$
Then we have
$$
\int_{B_R}u(x)^pdx\geq c\int_{B_R}(R+|x|)^{pp_1}dx(\int_{B_R}u^p)^p.
$$
Note that
$$\int_{B_R}(R+|x|)^{pp_1}dx=\int_{B_R}(R+|x|)^{-n}dx=c_1
$$
which is independent of $R$. Hence
$$
\int_{R^n}u(x)^pdx<\infty.
$$
Let $T_R=B_{2R}-B_R$. Then arguing as before,
$$
\int_{T_R}u(x)^pdx\geq c\int_{T_R}(R+|x|)^{-n}dx(\int_{B_R}u^p)^p=c_2(\int_{B_R}u^p)^p.
$$
Sending $R|to\infty$, we get that the left side goes to zero and then
$$
\int_{R^n}u(x)^pdx=0
$$
and $u=0$.

This completes the proof.
\end{proof}

We give two remarks. One is that
we may argue in the case $p=\frac{n}{n-\alpha}$ by the same method as in the case $1<p<\frac{n}{n-\alpha}$.
For $|x|\geq R/2$, we have
\begin{equation}\label{star}
u(x)\geq c|x|^{p_1}\int_{B_R}u^p.
\end{equation}
Then
$$
\int_{B_R}u(x)^p\geq c^p (\int_{B_R}u^p)^p\int_{B_R} dx\int_{R^n}\frac{(R+|y|)^{pp_1}}{|x-y|^{n-\alpha}}dy.
$$
This implies that
\begin{equation}\label{key5}
1\geq c^p (\int_{B_R}u^p)^{p-1}\int_{B_R} dx\int_{\{|y|\geq R/2\}}\frac{(R+|y|)^{-n}}{|x-y|^{n-\alpha}}dy
\end{equation}
Let
$$
I=\int_{B_R} dx\int_{\{|y|\geq R/2\}}\frac{(R+|y|)^{-n}}{|x-y|^{n-\alpha}}dy.
$$
Then
$$
I\geq cR^n\int_{\{|y|\geq R/2\}}\frac{(R+|y|)^{-n}}{(R+|y|)^{n-\alpha}}dy=c_0>0.
$$
Then from (\ref{key5}) we know that
$$
\int_{R^n}u^p<\infty.
$$
Then by (\ref{star}), we have
$$
\int_{T_R} u^p\geq  c^p (\int_{B_R}u^p)^p\int_{T_R} dx\int_{R^n}\frac{(R+|y|)^{-n}}{|x-y|^{n-\alpha}}dy.
$$
Compute the right side of above inequality as before we get that
$$
\int_{T_R} u^p\geq c(\int_{B_R}u^p)^p.
$$
Sending $R\to \infty$, we know that the left side approaches zero and we conclude that
$$
\int_{R^n}u^p=0.
$$
Then $u=0$.

The other is that the above argument works for $\rho(x)\geq (1+|x|)^{-\theta}$ with $\theta\in (0,\alpha^2/n)$ for $\alpha\in (0,2)$. We omit the detail.

\section{ Fujita exponent for nonlinear fractional Laplacian heat equation}\label{sect5}

Assume that $f:R_+\to R_+$ is a locally Lipschitz convex function with the derivative condition $f'(0)=0$. Let $T>0$.
Assume that $u=u(x,t)$ is a nonnegative mild solution to the nonlocal nonlinear fractional diffusion
\begin{equation}\label{evolution}
(\partial_t+L)u=f(u), \ \ in \ \ R^n\times (0,T)
\end{equation}
with non-trivial nonnegative Cauchy data $u(x,0)=u_0(x)$. The fractional operator $L$ is as in last section. In short we let $u(t)=u(x,t)$ when it is considered as a element in some Banach space. Let $0\leq t<T$ and
Let $G(x,y,t):=G_\alpha(x-y,t$ be the fundamental solution of the operator $\partial_t+L$ and define
$$
e^{-tL}w(x)=\int_{R^n}G(x,y,t)w(y)dy.
$$
Recall that the above expression is well-defined for $w\in L^1\bigcap L^\infty(R^n)$.

Let $\sigma=\alpha/2$. We now recall some standard fact about the kernel function $G_\alpha(x,t)$ \cite{BG}. We know that $G_\alpha(x,t)>0$ for $t>0$ and
$$
\int_{R^n}G_\alpha(x,t)dx=1.
$$
For any $s>0$, $t>0$, we have
$$
G_\alpha(\cdot,t)*G_\alpha(\cdot,s)=G_\alpha(\cdot,t+s).
$$
We have
$$
G_\alpha(x,t)=t^{-\frac{n}{\alpha}}G_\alpha(t^{-\frac{n}{\alpha}}x,1)
$$
and
$$
\lim_{|x|\to\infty}|x|^{n+\alpha}G_\alpha(x,1)=c_{n,\alpha}>0.
$$
By the last fact we can see that
there is dimensional constant $B=B_{n,\alpha}$ such that
\begin{equation}\label{asym}
\frac{B^{-1}}{t^{\frac{n}{\alpha}}(1+|t^{-\frac{1}{\alpha}}x|^2)^{\frac{n+\alpha}{2}}}\leq G_\alpha(x,t)\leq \frac{B}{t^{\frac{n}{\alpha}}(1+|t^{-\frac{1}{\alpha}}x|^2)^{\frac{n+\alpha}{2}}}
\end{equation}
In the particular case when $\alpha=1$, we have
$$
G_1(x,t)=\frac{B}{t^{n}(1+|t^{-1}x|^2)^{\frac{n+1}{2}}}=\frac{Bt}{(t^2+|x|^2)^{\frac{n+1}{2}}},
$$
where $B=\Gamma(\frac{n+1}{2})/\pi^{\frac{n+1}{2}}$. This is the standard Poisson kernel \cite{St}. 
Note that for any non-negative function $v$ and $R>0$,
$$
G_\alpha(\cdot,t)*v(x)=\int_{R^n}G_\alpha(x-y,t)v(y)dy\geq \int_{B_R}G_\alpha(x-y,t)v(y)dy.
$$
Then,
\begin{equation}\label{asym2}
\underline{\lim}_{t\to\infty}t^{\frac{n}{\alpha}}G_\alpha(\cdot,t)*v=\underline{\lim}_{t\to\infty}G_\alpha(t^{-\frac{n}{\alpha}}\cdot,1)\geq G(0,1)\int_{B_R}v
\end{equation}
for any ball $B_R$.
In fact with a little more effort one may show
that up to a constant scalar,
\begin{equation}\label{asym0}
\lim_{t\to\infty}t^{\frac{n}{\alpha}}G_\alpha(\cdot,t)*v=|v|_{L^1(R^n)}.
\end{equation}

We now use a trick from Weissler \cite{W81}.

Note that we have following expression for the solution $u(t)$ of (\ref{evolution})
$$
u(t)=e^{-tL}u_0+\int_0^te^{-(t-s)L}f(u(s))ds.
$$
Introduce the parameter $\tau>t$ and define
$$
H(x,t)=e^{-\tau L}u_0+\int_0^tf(e^{-(\tau-s)L}u(s))ds.
$$
Clearly that $H(x,0)=e^{-\tau L}u_0>0$.
Note that
$$
e^{-(\tau-t)L}u(t)=e^{-\tau L}u_0+\int_0^te^{-(\tau-s)L}f(u(s))ds.
$$
From Jensen's inequality we have
$$
e^{-(\tau-s)L}f(u(s))\geq f(e^{-(\tau-s)L}u(s)).
$$
Then we have
$$
e^{-(\tau-t)L}u(t)\geq H(x,t).
$$

Clearly we have
$$
H(x,t)\leq e^{-\tau L}u_0+\int_0^\tau f(e^{-(\tau-s)L}u(s))ds=u(x,\tau).
$$
For fixed $x\in R^n$ and using $f'\geq 0$ on $[0\infty)$, we have
\begin{equation}\label{star2}
\partial_tH(x,t)=f(e^{-(\tau-t)L}u(t))\geq f(H(x,t)).
\end{equation}

We now consider the special case when $f(u)=u^p$ with $p>1$.
The inequality (\ref{star2}) gives us
$$
\partial_tH(x,t)\geq H(x,t)^p.
$$
Integrating in $t$, we have
$$
H(x,0)^{1-p}\geq H(x,\tau)^{1-p}+(p-1)\tau\geq (p-1)\tau
$$
and then the very useful an priori bound
\begin{equation}\label{star3}
(p-1)^{1/(p-1)}\tau^{1/(p-1)} H(x,0)\leq 1.
\end{equation}
Recall that $H(x,0)=H(x,0)=e^{-\tau L}u_0(x)$. Then (\ref{star3}) implies that
\begin{equation}\label{star4}
(p-1)^{1/(p-1)}\tau^{1/(p-1)}|e^{-\tau L}u_0|_{L^\infty(R^n)} \leq 1.
\end{equation}
This is a very interesting trick since we use nonlinear equation to get exact $L^\infty$ bound for the Cauchy problem of the linear fractional evolution equation
\begin{equation}\label{evolution2}
(\partial_t+L)u=0, \ \ in \ \ R^n\times (0,T)
\end{equation}
with nonnegative Cauchy data $u(x,0)=u_0(x)$.

The following result asserts that the exponent $p_F:=\frac{n+\alpha}{n}$ is the Fujita exponent to the evolution problem (\ref{evolution3}).

\begin{Thm}\label{fujita}
Assume that $1<p\leq 1+\frac{\alpha}{n}$. Then there is no global solution to
the nonlocal nonlinear fractional diffusion
\begin{equation}\label{evolution3}
(\partial_t+L)u=u^p, \ \ in \ \ R^n\times (0,T)
\end{equation}
with non-trivial nonnegative Cauchy data $u(x,0)=u_0(x)$.
\end{Thm}
\begin{proof} Assume that $p\in (1,p_F)$. Since
$$
\frac{n}{\alpha}<1/(p-1),
$$
 the estimate (\ref{star4}) implies that
$$
\lim_{t\to\infty}t^{\frac{n}{\alpha}}|G_\alpha (\cdot, t)u_0|_{L^\infty(R^n)}=0,
$$
which contradicts with (\ref{asym2}).

  Let $p=p_F$. We denote by $G_t(\cdot)=G_\alpha(\cdot,t)$ and $P_tu=G_\alpha(\cdot,t)*u$. Using the semigroup property of $G_t$ we know that for any $t_0>0$ and $t>0$,
  $$
  u(t+t_0)=P_tu(t_0)+\int_0^tP_{t-s}u(s+t_0)^pds.
  $$
Since $p-1=\frac{\alpha}{n}$, by (\ref{asym2}) and (\ref{star4}) we know that $u(t)$ is in $L^1(R^n)$ and there is a uniform constant $C>0$ such that
$$
|u(t)_{L^1(R^n)}\leq C.
$$
We shall see that this is impossible for $(p-1)n=\alpha$ via a direct computation.

Using
$$
|t^{-\frac{1}{\alpha}}(x-y)|^2\leq 2(|t^{-\frac{1}{\alpha}}x|^2+|t^{-\frac{1}{\alpha}}y|^2),
$$
we know from (\ref{asym}) that for $$
k(t)=c_0\int_{R^n}\frac{B^{-1}}{(1+|z|^2)^{\frac{n+\alpha}{2}}}u_0(t^{\frac{1}{\alpha}}z)dzt^{\frac{n}{\alpha}}
$$ with $c_0>0$ being a uniform constant,
$$
u(x,t)\geq P_t u_0(x)\geq G_{t}(x)k(t).
$$
Then $k=k(1)$,
$$
u(x,1)\geq kG_1(x)
$$
and for any $s>0$,
$$
u(s+1)\geq G_s*u(1)\geq k(1)G_s*G_1=kG_{s+1}.
$$
By (\ref{asym}) and a direct computation we know that there is a uniform constant $C_2>0$ such that
$$
|G_{s+1}^p|_{L^1(R^n)}\geq \frac{C_2}{s+1}.
$$
Then,
\begin{align*}
|u(t+2)|_{L^1(R^n)}&\geq \int_0^t|G_{t-s}*u^p(s+2)|_{L^1(R^n)}ds \\
&\geq \int_0^t|G_{t-s}*(k(s+2)G_{s+2})^p|_{L^1(R^n)}ds \\
&=\int_0^tk(s+2)^p|G_{s+2}^p|_{L^1(R^n)}ds\\
&\geq C_2\int_0^t(s+1)^{-1}ds\to\infty
\end{align*}
as $t\to\infty$.
This is absurd.
This completes the proof of Theorem \ref{fujita}.
\end{proof}

\section{Conclusions}\label{sect6}
All results in our paper are new. We have
made the Fujita exponent for the space-diffusion nonlinear fractional evolution problem (\ref{diffusion2}). Since in the $\alpha$-Laplacian operator case, the kernels behavior very diffrent from the standard Gaussian kernel, the result is surprising. Though we have followed the idea from Weissler's paper \cite{W81}, the details are different in essential. We have developed the Liouville theorem for (\ref{ell}) when $0<\alpha<2$ and $0<p<\frac{n}{n-\alpha}$ and the details we have presented are very readable for applied analysts, physicist, and engineers. As is well-known, One can use blow-up argument in \cite{PQS} and Liouville type results to get an priori bounds for positive solutions of related nonlinear fractional elliptic problems in bounded domains such as
\begin{equation}\label{ell2}
(-\Delta)^{\alpha/2} u=\rho(x)u^p+f(x) \ \ \ in \ \ \Omega\subset \subset R^n,
\end{equation}
where $0<p\leq p_{sg}$.

We think that our existence result about the problem
\begin{equation}\label{beta}
(-\Delta)^{\alpha/2} u=\rho(x)u^\beta, \ \ \ in \  R^n,
\end{equation}
where $\beta=1/2$, is a beginning point for a general study of the case when $\alpha\in (0,1)$. Assume that $f=0$ in (\ref{diffusion}) and $h(u)=u^\beta$ with $\beta>0$. Let $u(x,t)=w(x)\phi(t)$. We are led to
\begin{equation}\label{beta2}
(-\Delta)^{\alpha/2} w=\lambda w^\beta, \ \ in \ \ R^n
\end{equation}
and
$$
D_{0+}^\tau \phi(t)^\beta=-\lambda \phi(t), \ \ t>0
$$
for some real constant $\lambda$. equation (\ref{beta}) is a special case of (\ref{beta2}) when $\lambda=1$. For arbitrary $\beta\in (0,1)$, we know very few result about about the nonlocal nonlinear eigenvalue problem (\ref{beta2}).

One may also consider related existence result about (\ref{ell}) with the nonlocal nonlinear term $f(u)=(K*u^2)u$ just as in the Choquard-Hartree equation \cite{MZ} and we leave it open. One may also consider finite blow up for space-fractional diffusion half-space with nonlinear boundary condition \cite{Hu} or the blow up rate for space-fractional diffusion on bounded domain \cite{FS}.

\end{document}